\documentclass[12pt]{amsart}
\usepackage{amssymb}
\usepackage{mathrsfs}
\usepackage{amsxtra}
\usepackage{graphics}
\usepackage{latexsym}
\usepackage{xcolor}
\usepackage{amsmath}
\usepackage{amssymb,amsthm,amsfonts}
\usepackage{mathtools}
\usepackage{amscd}
\usepackage{tikz-cd}
\usepackage[arrow, matrix, curve]{xy}
\usepackage{syntonly}
\title[]{Nonlinear harmonic bundles}

\author[Mao Sheng]{Mao Sheng}
\email{msheng@tsinghua.edu.cn}
\address{Yau Mathematical Science Center, Tsinghua University, Beijing, 100084, China}
\address{Yanqi Lake Beijing Institute of Mathematical Sciences and Applications, Beijing, 101408, China}

\begin{document}
\theoremstyle{plain}
\newtheorem{thm}{Theorem}[section]
\newtheorem{theorem}[thm]{Theorem}
\newtheorem{lemma}[thm]{Lemma}
\newtheorem{corollary}[thm]{Corollary}
\newtheorem{proposition}[thm]{Proposition}
\newtheorem{addendum}[thm]{Addendum}
\newtheorem{variant}[thm]{Variant}
\theoremstyle{definition}
\newtheorem{lemma and definition}[thm]{Lemma and Definition}
\newtheorem{construction}[thm]{Construction}
\newtheorem{statement}[thm]{Statement}
\newtheorem{notations}[thm]{Notations}
\newtheorem{question}[thm]{Question}
\newtheorem{problem}[thm]{Problem}
\newtheorem{remark}[thm]{Remark}
\newtheorem{remarks}[thm]{Remarks}
\newtheorem{definition}[thm]{Definition}
\newtheorem{claim}[thm]{Claim}
\newtheorem{assumption}[thm]{Assumption}
\newtheorem{assumptions}[thm]{Assumptions}
\newtheorem{properties}[thm]{Properties}
\newtheorem{example}[thm]{Example}
\newtheorem{conjecture}[thm]{Conjecture}
\newtheorem{proposition and definition}[thm]{Proposition and Definition}
\numberwithin{equation}{thm}
\newcommand{\Spec}{\mathrm{Spec}}
\newcommand{\pP}{{\mathfrak p}}
\newcommand{\sA}{{\mathcal A}}
\newcommand{\sB}{{\mathcal B}}
\newcommand{\sC}{{\mathcal C}}
\newcommand{\sD}{{\mathcal D}}
\newcommand{\sE}{{\mathcal E}}
\newcommand{\sF}{{\mathcal F}}
\newcommand{\sG}{{\mathcal G}}
\newcommand{\sH}{{\mathcal H}}
\newcommand{\sI}{{\mathcal I}}
\newcommand{\sJ}{{\mathcal J}}
\newcommand{\sK}{{\mathcal K}}
\newcommand{\sL}{{\mathcal L}}
\newcommand{\sM}{{\mathcal M}}
\newcommand{\sN}{{\mathcal N}}
\newcommand{\sO}{{\mathcal O}}
\newcommand{\sP}{{\mathcal P}}
\newcommand{\sQ}{{\mathcal Q}}
\newcommand{\sR}{{\mathcal R}}
\newcommand{\sS}{{\mathcal S}}
\newcommand{\sT}{{\mathcal T}}
\newcommand{\sU}{{\mathcal U}}
\newcommand{\sV}{{\mathcal V}}
\newcommand{\sW}{{\mathcal W}}
\newcommand{\sX}{{\mathcal X}}
\newcommand{\sY}{{\mathcal Y}}
\newcommand{\sZ}{{\mathcal Z}}
\newcommand{\A}{{\mathbb A}}
\newcommand{\B}{{\mathbb B}}
\newcommand{\C}{{\mathbb C}}
\newcommand{\D}{{\mathbb D}}
\newcommand{\E}{{\mathbb E}}
\newcommand{\F}{{\mathbb F}}
\newcommand{\G}{{\mathbb G}}
\renewcommand{\H}{{\mathbb H}}
\newcommand{\I}{{\mathbb I}}
\newcommand{\J}{{\mathbb J}}
\renewcommand{\L}{{\mathbb L}}
\newcommand{\M}{{\mathbb M}}
\newcommand{\N}{{\mathbb N}}
\renewcommand{\P}{{\mathbb P}}
\newcommand{\Q}{{\mathbb Q}}
\newcommand{\Qbar}{\overline{\Q}}
\newcommand{\R}{{\mathbb R}}
\newcommand{\SSS}{{\mathbb S}}
\newcommand{\T}{{\mathbb T}}
\newcommand{\U}{{\mathbb U}}
\newcommand{\V}{{\mathbb V}}
\newcommand{\W}{{\mathbb W}}
\newcommand{\Z}{{\mathbb Z}}
\newcommand{\g}{{\gamma}}
\newcommand{\id}{{\rm id}}
\newcommand{\rk}{{\rm rank}}
\newcommand{\END}{{\mathbb E}{\rm nd}}
\newcommand{\End}{{\rm End}}
\newcommand{\Hom}{{\rm Hom}}
\newcommand{\Hg}{{\rm Hg}}
\newcommand{\tr}{{\rm tr}}
\newcommand{\Sl}{{\rm Sl}}
\newcommand{\GL}{{\rm Gl}}
\newcommand{\Cor}{{\rm Cor}}

\newcommand{\SO}{{\rm SO}}
\newcommand{\OO}{{\rm O}}
\newcommand{\SP}{{\rm SP}}
\newcommand{\Sp}{{\rm Sp}}
\newcommand{\UU}{{\rm U}}
\newcommand{\SU}{{\rm SU}}
\newcommand{\SL}{{\rm SL}}
\newcommand{\ra}{\rightarrow}
\newcommand{\la}{\leftarrow}
\newcommand{\Gal}{\mathrm{Gal}}
\newcommand{\Res}{\mathrm{Res}}
\newcommand{\Gl}{\mathrm{Gl}}
\newcommand{\Gr}{\mathrm{Gr}}
\newcommand{\Exp}{\mathrm{Exp}}
\newcommand{\Sym}{\mathrm{Sym}}
\newcommand{\Ann}{\mathrm{Ann}}
\newcommand{\GSp}{\mathrm{GSp}}
\newcommand{\Tr}{\mathrm{Tr}}
\newcommand{\HIG}{\mathrm{HIG}}
\newcommand{\MIC}{\mathrm{MIC}}
\newcommand{\FV}{\mathrm{FV}}
\newcommand{\HV}{\mathrm{HV}}
\newcommand{\Ext}{\mathrm{Ext}}
\newcommand{\bA}{\mathbf{A}}
\newcommand{\bK}{\mathbf{K}}
\newcommand{\bM}{\mathbf{M}} 
\newcommand{\bP}{\mathbf{P}}
\newcommand{\bC}{\mathbf{C}}
\newcommand{\NMF}{\mathrm{NMF}}
\newcommand{\sFV}{\mathrm{SFV}}
\newcommand{\sHV}{\mathrm{SHV}}
\newcommand{\Aut}{{\rm Aut}}
\newcommand{\Gm}{{\rm Gm}}

\maketitle
\begin{abstract}
We generalize the notion of harmonic bundles in nonabelian Hodge theory to the nonlinear setting.
\end{abstract}
\section{Introduction}
In nonabelian Hodge theory, flat bundles and Higgs bundles are interconnected via harmonic metrics, which is the core idea of the notion of harmonic bundles. Can we still have a meaningful theory if one replaces a typical fiber, which is a finite dimensional complex vector space, by a complex manifold? This question is not totally absurd, because we have recently obtained a Hitchin-Simpson type correspondence in positive characteristic \cite{S}, that generalizes the nonabelian Hodge correspondence for vector bundles in positive characteristic, as established by Ogus-Vologodsky \cite{OV}. Now we are looking for a complex analogue of \cite{S}. The present note presents our preliminary idea about a possible nonlinear Hodge theory over the field of complex numbers, and the main purpose here is to introduce the notion of a \emph{nonlinear harmonic bundle}, generalizing the notion of a harmonic bundle in nonabelian Hodge theory.\\

We are thinking of differentiable fiber bundles, equipped with transversal foliations, holomorphic with respect to some complex structures on the fiber bundles, as well as Higgs fields (defined below), holomorphic with respect to possibly another complex structures on the same fiber bundles, which are reconstructable from each other by certain mechanism involving hermitian metrics on the fiber bundles. The precise meaning of it is the main content of the note.  Existence of harmonic metrics are related to slope stability of holomorphic vector bundles. Likewise, existence of nonlinear harmonic metrics should be related to certain nonlinear slope stability on holomorphic fibrations.  \\

Conjecturally, relative nonabelian Hodge moduli spaces provide examples of nonlinear harmonic bundles. We show it is indeed the case for rank one moduli. As one merit of our considerations, the nonlinear Higgs fields on the rank one Higgs moduli spaces determine the integral structure of weight one variations of Hodge structure with Zariski dense monodromy, which is inaccessible to Higgs fields in nonabelian Hodge theory. \\

{\bf Acknowledgement.} This note benefits a lot from numerous discussions with Nianzi Li and Zhaofeng Yu. Nianzi Li corrected the expression of $\omega_U$ in Example 3.3 by pointing out my original formula does not glue globally. He also pointed out that my early conditions on $\beta$ will only lead to the identity map. The inclusion of a reasonable $\beta$ into the Simpson mechanism really came from a discussion with Zhaofeng Yu, when he asked me whether we would recover the nonabelian Hodge correspondence from the Simpson mechanism as I originally proposed. This work is partially supported by the Chinese Academy of Sciences Project for Young Scientists in Basic Research (Grant No. YSBR-032). 
 
\section{Almost connections and almost Higgs fields}
In this note, the notation $S$ is reserved for a complex manifold $S$ as the base manifold. Let $\alpha: Z\to S$ be a differentiable fiber bundle over the complex manifold $S$. Properness of $\alpha$ is not assumed. Associated to $\alpha$ is the following short exact sequence of complex vector bundles over $S$
\begin{eqnarray}
0\to T_{Z_{\C}/S}\to T_{Z_{\C}}\to \alpha^*T_{S_{\C}}\to 0,    
\end{eqnarray}
where $T_{Z_{\C}/S}$ is the complexification of the real relative tangent bundle $T_{Z_{\R}/S}$ and etc. An integrable complex structure on $T_{Z_{\R}/S}$ is a choice of complex subbundle $T_{rel}\subset T_{Z_{\C}/S}$
satisfying $T_{Z_{\C}/S}=T_{rel}\oplus \bar T_{rel}$ and $[T_{rel},T_{rel}]\subset T_{rel}$. 
\begin{definition}\label{complex fiber bundle}
A complex fiber bundle over $S$ is a pair $(\alpha,T_{rel})$, where $\alpha: Z\to S$ is a differentiable fiber bundle over $S$ and $T_{rel}$ is an integrable complex structure on $T_{Z_{\R}/S}$.      
\end{definition}
The real relative tangent bundle restricts to each fiber of $\alpha$ the real tangent bundle of the fiber. Hence, a choice of $T_{rel}$ as above yields a complex structure on the fiber. So we may regard a complex fiber bundle as a differentiable family of complex manifolds. It generalizes naturally the notion of a complex vector bundle. It is well known that a holomorhic structure on a complex vector bundle over $S$ is equivalent to an integrable $\bar \partial$-operator on the bundle. Motivated by this fact, we define a $\bar \partial$-operator on the complex fiber bundle $(\alpha,T_{rel})$ to be a differentiable bundle morphism 
$$
\bar \partial: \alpha^*\bar T_S\to \frac{T_{Z_{\C}}}{\bar T_{rel}}
$$
whose image under the projection $\frac{T_{Z_{\C}}}{\bar T_{rel}}\to \alpha^*T_S\oplus \alpha^*\bar T_S$ is contained in $\alpha^*\bar T_S$ and such that the composite is the identity. Note that one has a short exact sequence of complex vector bundles attached to $(\alpha,T_{rel})$:
\begin{eqnarray}
0\to T_{rel}\to \frac{T_{Z_{\C}}}{\bar T_{rel}}\to \alpha^*T_S\oplus \alpha^*\bar T_S\to 0.    
\end{eqnarray}
A $\bar \partial$-operator is nothing but a splitting of the natural projection $\frac{T_{Z_{\C}}}{\bar T_{rel}}\to \alpha^*\bar T_S$. 
\begin{definition}\label{almost complex structure}
An almost complex structure on the complex fiber bundle $(\alpha,T_{rel})$ is a complex subbundle $T\subset T_{Z_{\C}}$ which contains $T_{rel}$ and satisfies 
\begin{itemize}
    \item [(1)] $T\oplus \bar T= T_{Z_{\C}}$;
    \item [(2)] The image of $T$ under the composite $T\subset T_{Z_{\C}} \to \alpha^*T_{S_{\C}}$ is contained in $\alpha^*T_S$.
\end{itemize}
If $T$ satisfies further the condition
\begin{itemize}
\item [(3)] $[T,T]\subset T$,
\end{itemize}
then we call it a complex structure. 
\end{definition}
By the theorem of Newlander-Nirenberg, a complex structure on $(\alpha,T_{rel})$ gives rise to a holomorphic fibration structure on $\alpha$ whose holomorphic relative tangent bundle equals $T_{rel}$, and vice versa. 
\begin{proposition}\label{almost cmoplex structures in terms of d bar operators}
Let $(\alpha, T_{rel})$ be a complex fiber bundle over $S$. Then a complex structure on $(\alpha, T_{rel})$ gives rise to a canonical $\bar \partial$-operator on $(\alpha, T_{rel})$. Conversely, a $\bar \partial$-operator on $(\alpha, T_{rel})$ gives a canonical almost complex structure on $(\alpha, T_{rel})$. It is a complex structure if and only if the inverse image of $\bar \partial(\alpha^*\bar T_S)\subset \frac{T_{Z_{\C}}}{\bar T_{rel}}$ in $T_{Z_{\C}}$ is closed under Lie bracket (we call $\bar \partial$ integrable if it occurs). 
\end{proposition}
\begin{proof}
For clarity of notation, we shall write $f: X\to S$ for the holomorphic fibration over $S$ corresponding to a given complex structure $T$ on $(\alpha,T_{rel})$. Thus the underlying differentiable fiber structure of $f$ is just $\alpha$ and $T_{X/S}=T_{rel}$ as complex subbundle of $T_{Z_{\C}/S}$. Let $x\in X$ and $s=f(x)\in S$. Locally around $s$, one may take a set of local coordinates $\{s_1,\cdots,s_m\}$ of $S$ such that $\{s_1\circ \alpha, \cdots, s_m \circ \alpha\}$ extends to a set of local coordinates $\{s_1\circ \alpha, \cdots, s_m \circ \alpha,t_1,\cdots,t_r\}$ of $X$ around $x$. By abuse of notation, we shall write $s_i\circ \alpha$ simply by $s_i$, and thereby understand $f$ as the natural projection 
$$
(s_1,\cdots,s_m,t_1,\cdots,t_r)\mapsto (s_1,\cdots,s_m).
$$
Then $\{\partial_{s_1},\cdots,\partial_{s_m}\}$ form a holomorphic basis of $T_S$ on an open neighborhood $V$ of $s$, and $\{\partial_{s_1},\cdots,\partial_{s_m},\partial_{t_1},\cdots,\partial_{t_r}\}$ forms a holomorphic basis of $T_X$ on an open neighborhood $U\subset \alpha^{-1}(V)$ of $x$, in which $\{\partial_{t_1},\cdots,\partial_{t_r}\}$ forms a holomorphic basis of $T_{X/S}$ over $U$. Define $\bar \partial$ over $U$ by the association
$$
\bar \partial_{s_i}\mapsto \bar \partial_{s_i}\ \textrm{mod}\ \bar T_{X/S}.
$$
Suppose $x$ lies in another open subset $U'\subset f^{-1}V$ with a set of new local coordinates $\{s_1,\cdots, s_m, t'_1,\cdots,t'_r\}$, where for each $i$,
$$
t'_i=t'_i(s_1,\cdots,s_m,t_1,\cdots,t_r)
$$
is holomorphic in each variable. By the chain rule, $\bar \partial_{s_i}$ in the new coordinates is mapped to 
$$
\bar \partial_{s_i}+\sum_{1\leq j\leq r}\overline{(\frac{\partial t'_j}{\partial s_i})}\bar \partial_{t'_j},
$$
which is equal to $\bar \partial_{s_i}$ modulo $\bar T_{X/S}$. Suppose $s=f(x)\in V\subset S$ lies in another open subset $V'$ with a set of new local coordinates $\{s'_1,\cdots,s'_m\}$. Then the association 
$$
\bar \partial_{s'_i}\mapsto \bar \partial_{s'_i}\ \textrm{mod}\ \bar T_{X/S}
$$
coincide with the previous association, because $s'_i=s'_i(s_1,\cdots,s_m)$ is holomorphic in each variable. Hence, we obtain a well-defined $\bar \partial$-operator on $(\alpha,T_{rel})$. This completes the first part. Conversely, given a $\bar \partial$-operator on $(\alpha,T_{rel})$, we denote the inverse image of $\textrm{im}(\bar \partial)$ under the natural projection $T_{Z_{\C}}\to \frac{T_{Z_{\C}}}{\bar T_{rel}}$ by $\bar T$. Then the rank of $\bar T$ is equal to 
$$\rk \ \textrm{im}(\bar \partial)+\rk \ \bar T_{rel}=\rk\ \bar T_S+\rk\ \bar T_{rel}.$$ 
Moreover, since the image of $T$, the complex conjugation of $\bar T$, under the natural projection $T_{Z_{\C}}\to \alpha^*T_{S_{\C}}$ is contained in $\alpha^*T_S$, an element in $T\cap \bar T$ has to be contained in $T_{rel}\cap \bar T_{rel}$, which is the zero. Thus, for reason of rank, we have a decomposition $T_{Z_{\C}}=T\oplus \bar T$. As $T$ contains $T_{rel}$ by definition, it defines an almost complex structure on $(\alpha,T_{rel})$. By definition, $\bar\partial$ is integrable if and only if $[\bar T,\bar T]\subset \bar T$, equivalently $[T,T]\subset T$ by taking the complex conjugation. This completes the proof.  
\end{proof}
When $(\alpha, T_{rel})$ is a complex vector bundle, a $\bar \partial$-operator on it in the above sense is much more general than a $\bar \partial$-operator in the classical sense. For example, while the integrability of a classical $\bar \partial$-operator can be measured through the vanishing of a curvature tensor valued in $\wedge^2\bar \Omega_S\otimes \alpha_*T_{rel}$, it is not the case for an arbitrary $\bar \partial$-operator. We shall now discuss a condition for $\bar \partial$ so that the integrability of the corresponding almost complex structure is measured through the vanishing of a curvature tensor, generalizing the complex vector bundle situation. 
\begin{lemma}\label{integrability of almost complex structure}
Let $(\alpha, T_{rel})$ be a complex fiber bundle on $S$. Consider the set of $\bar \partial$-operators on $(\alpha, T_{rel})$ satisfying the following local lifting condition: For any point $x\in Z$, there is a local lifting $\tilde {\bar \partial}: \alpha^*\bar T_S\to T_{Z_{\C}}$ of $\bar \partial$ around $x$ so that 
$$
[\bar T_{rel},\textrm{im}(\tilde{\bar \partial})]\subset \bar T_{rel}.
$$
Then there is a differentiable bundle map $F^{0,2}: \wedge^2 \bar T_S\to \alpha_*\frac{T_{Z_{\C}}}{\bar T}$ such that it is zero if and only if its corresponding almost complex structure $\bar T$ on $(\alpha, T_{rel})$ is integrable. 
\end{lemma}
Since $\bar T_{rel}$ is closed under Lie bracket, the above condition is independent of the choice of such a local lifting. 
\begin{proof}
By definition of $\bar T$, there is a short exact sequence of complex vector bundles:
$$
0\to \bar T_{rel}\to \bar T\to \alpha^*\bar T_S\to 0.
$$
A local lifting of $\bar \partial$ is in fact nothing but a local splitting of the projection $\bar T\to \alpha^*\bar T_S$. Write locally $\bar T=\bar T_{rel}\oplus \textrm{im}(\tilde{\bar \partial})$. Then we see that, by the condition satisfied by $\tilde{\bar \partial}$ in the lemma, $[\bar T,\bar T]\subset \bar T$ holds if and only if $[\textrm{im}(\tilde{\bar \partial}),\textrm{im}(\tilde{\bar \partial})]\subset \bar T$ holds, equivalently, the composite 
$$
\alpha^*\wedge^2\bar T_S\stackrel{[\tilde{\bar \partial},\tilde{\bar \partial}]}{\longrightarrow} T_{Z_{\C}}\to \frac{T_{Z_{\C}}}{\bar T}, \quad a\wedge b \mapsto \overline{[\tilde{\bar \partial}(a),\tilde{\bar \partial}(b)]},
$$
is zero. As remarked above, although $\tilde{\bar \partial}$ is only locally defined, the previous composite map is actually globally defined. We shall define $F^{0,2}$ to be its adjoint map. That concludes the proof.
\end{proof}

The $\bar \partial$-operator on $(\alpha,T_{rel})$ associated to a holomorphic fibration $f: X\to S$ will be denoted by $\bar \partial_f$ below. Associated to $f$ is the following short exact sequence of holomorphic vector bundles over $X$:
\begin{eqnarray}
0\to T_{X/S}\to T_X\stackrel{\pi}{\to}f^*T_S\to 0.
\end{eqnarray}
A holomorphic transversal foliation on $f$ is a holomorphic subbundle $\sF\subset T_X$ which is closed under the Lie bracket and projects isomorphically onto $f^*T_S$. Equivalently, it is a holomorphic splitting of the above sequence, which is given by an $\sO_X$-linear morphism $\nabla: f^*T_S\to T_X$ whose composite with $\pi$ is the identity, such that its image is closed under the Lie bracket. Another useful point of view is that a holomorphic transversal foliation on $f$ gives rise to a first order differential operator on the sheaf of local holomorphic cross sections of $f$, and vice versa. When $f$ is a principal $G$-bundle where $G$ is a connected complex Lie group, a $G$-equivariant holomorphic splitting $\nabla$ is nothing but a holomorphic $G$-connection on $f$ (\cite{At}). 
\begin{definition}\label{holomorphic connection}
Let $f: X\to S$ be a holomorphic fibration over $S$. A holomorphic connection on $f$ is a holomorphic bundle morphism $\nabla: f^*T_S\to T_X$ which splits the natural projection $T_X\to f^*T_S$. It is said to be integrable or flat, if the image $\nabla(f^*T_S)\subset T_X$ is closed under Lie bracket, equivalently, its curvature
$$
F^{2,0}: \wedge^2T_S \longrightarrow f_*T_{X/S}, \quad a\wedge b\mapsto \overline{[\nabla(a),\nabla(b)]}
$$
vanishes. 
\end{definition}
A holomorphic connection does not necessarily exist, evenly locally in $S$. When $f$ is proper and $\dim S=1$, the existence of a holomorphic connection on $f$ forces $f$ to be isotrivial, viz, locally a product. However, there are plenty of interesting examples of flat holomorphic connections which are \emph{nonlinear} in nature. 
\begin{example}\label{representations of fundamental groups}
For a pair of positive integers $(m,n)$, one considers representations $\rho: \pi_1(\C-\{1,\cdots,m\},0)\to \Aut(\C^n)$, where $\Aut(\C^n)$ is the group of holomorphic automorphisms of $\C^n$. Let $S=\C-\{1,\cdots,m\}$ and $\tilde S$ be the universal cover of $S$. Over $\tilde S$, we have the trivial family $\tilde f: \tilde X:=\tilde S\times \C^n\to \tilde S$, equipped with the trivial transveral foliation on $\tilde f$. Let $\pi_1(S)$ act on $\tilde S$ by the deck transformation (right action) and on $\C^n$ via the representation $\rho$ (left action). Two pairs $(s_i,t_i)\in \tilde S\times \C^n, i=1,2$ is in the same orbit if there is an element $\gamma\in \pi_1(S)$ such that 
$$
(s_1,t_2)=(s_1\gamma^{-1},\gamma t_1). 
$$
Passing to the quotient by $\pi_1(S)$, the family $\tilde f$ as well as the transversal foliation induce a holomorphic fiber bundle $f_{\rho}: X_{\rho}=\tilde X/\pi_1(S)\to \tilde S/\pi_1(S)=S$, and a holomorphic flat connection $\nabla_{\rho}: f_{\rho}^*T_S\to T_{X_\rho}$. Unless $n=1$, a general element in $\Aut(\C^n)$ is not linear: Indeed, the subgroup $\Aut_{alg}(\C^n)$ consisting of algebraic automorphisms of $\C^n$ is not an algebraic group for $n\geq 2$, and is not even dense in the whole group (e.g $(t_1,\cdots,t_n)\mapsto (t_1,\cdots, e^{t_1}t_n)$ cannot be approximated by elements in $\Aut_{alg}(\C^n)$). Unlike in the linear case, the monodromy group of $\rho$ may or may not be contained in a complex Lie subgroup of $\Aut(\C^n)$, which indicates the complexity of the associated connection $\nabla_{\rho}$. Take $m=2$ and $n=2$ as example. After fixing a set of generators of $\pi_1(\C-\{1,2\},0)$, a representation $\rho$ is given by two elements of $\Aut(\C^2)$. Let $\rho_1$ be the one given by
$$
(t_1,t_2)\mapsto (t_1,t_1+t_2),\quad (t_1,t_2)\mapsto (t_1,t_1^2+t_2),
$$
and $\rho_2$ given by
$$
(t_1,t_2)\mapsto (t_2,t_1),\quad (t_1,t_2)\mapsto (t_1,t_1^2+t_2).
$$
It is clear that $\textrm{im}\rho_1$ is contained in the two dimensional complex Lie subgroup 
$$
\{(a_1,a_2)\in \C^2| (t_1,t_2)\mapsto (t_1,a_1t_1+a_2t_1^2+t_2)\}.
$$
Suppose it is also the case for $\rho_2$. As $\textrm{im}\rho_2\subset \Aut_{alg}(\C^2)$, it is contained a complex Lie subgroup $G$ of $\Aut_{alg}(\C^2)$, whose Lie algebra is 
$$
H^0(\A^2_{\C},T_{\A^2_{\C}/\C})= \C[t_1,t_2]\otimes_{\C}\C\{\partial_{t_1},\partial_{t_2}\}.
$$
Since the Lie algebra $\mathfrak{g}$ of $G$ is of finite dimension, the coefficients of its elements in $\C[t_1,t_2]\otimes_{\C}\C\{\partial_{t_1},\partial_{t_2}\}$ have bounded degree. A simple calculation on the Jacobian matrices of elements in $\textrm{im}\rho_2$ shows that it is not the case. Contradiction and therefore $\textrm{im}\rho_2$ is not contained in any complex Lie subgroup of $\Aut(\C^2)$. It is interesting to know how some of these examples can be related to the next set of examples in one-one correspondence manner.      
\end{example}
\begin{example}\label{rational ODE}
A rational ODE of order $n$ is of form 
$$
t^{(n)}=F(s,t,t^{(1)},\cdots,t^{(n-1)}),
$$
where $F$ is a polynomial in $(t^{(1)},\cdots,t^{(n-1)})$ with coefficients rational functions in variable $(s,t)$. The first order rational ODE was studied by Poincar\'e in his algebraically integrability problem on
rational foliations on $\C^2$ (that remains widely open). Set $t_i=t^{(i)}, 0\leq i\leq n-1$. One associates canonically a holomorphic flat connection on a holomorphic fibration $f: U\subset S\times \C^n \stackrel{pr}{\to}S$, where $S\subset \C$ (resp. $U\subset S\times \C^n$) is Zariski open, by the following formula:
$$
\partial_s\mapsto \partial_s+\sum_{0\leq i\leq n-2}t_{i+1}\partial_{t_i}+F(s,t_0,\cdots,t_{n-1})\partial_{t_{n-1}}.
$$
A (local) solution of the ODE $t=f(s)$ with the initial data at $s_0$ with
$$(f(s_0),f^{(1)}(s_0),\cdots,f^{(n-1)}(s_0))=(t_0,\cdots,t_{n-1})$$ gives rise to a (local) integral curve of the above tranversal foliation passing through the point $(s_0,t_0,\cdots,t_{n-1})\in S\times \C^n$, and vice versa. Unlike the linear case (when $F$ is linear in $(t,t^{(1)},\cdots,t^{(n)})$), the analytic continuation along a loop will not give rise to a monodromy representation of $\pi_1(S)$ in general. So it is an interesting problem to characterize those rational ODEs whose associated holomorphic flat connections come from Example \ref{representations of fundamental groups}. A famous example in this range is the Painlev\'e XI, a rational ODE of order two. 
\end{example}
It is necessary and actually important to consider a more general notion of connections than holomorphic connections.
\begin{definition}\label{almost connection}
An almost connection $\partial$ on a complex fiber bundle $(\alpha,T_{rel})$ is a differentiable splitting of the natural projection
$$
\frac{T_{Z_{\C}}}{\bar T_{rel}}\to \alpha^*T_S. 
$$ 
Let $(\alpha,T)$ be a complex fiber bundle equipped with an almost complex structure $T$ containing $T_{rel}$. An almost connection on $(\alpha,T)$ is an almost connection on $(\alpha,T_{rel})$ such that its image $\partial(\alpha^*T_S)$ is contained in $T\subset \frac{T_{Z_{\C}}}{\bar T_{rel}}$.
\end{definition}
We shall see natural examples of almost connections in the next section. Now we come back to the holomorphic fibration $f: X\to S$ as above. 
\begin{definition}\label{holomorphic Higgs field}
A holomorphic Higgs field on $f$ is an $\sO_S$-linear morphism $\theta: T_S\to f_*T_{X/S}$ satisfying the integrability condition $[\theta,\theta]=0$.   
\end{definition}
When $f$ is a holomoprhic vector bundle over $S$, the usual Higgs fields on the vector bundle, as first introduced by Hitchin in \cite{H}, are those linear holomorphic Higgs fields on $f$. For an arbitrary $f$, holomorphic Higgs fields on $f$ always exist, since one may simply take $\theta=0$.
\begin{example}
Let $S$ be a Riemann surface. Over $S$, there is a tautological linear Higgs bundle of rank two
$$
(f: X=\Omega_S\oplus \C_S \to S,\theta),
$$
where $\C_S=\C\times S$ is the constant bundle, and 
$$\theta: T_S\to f_*T_{X/S}\cong  \Sym(T_S\oplus \C_S)\otimes (\Omega_S\oplus \C_S)$$ is determined by 
$$
id: \Omega_S \to \C_S\otimes \Omega_S,\quad \C_S\to 0.
$$
For a positive integer $n$, put $[n]: \C_S\to \C_S$ to be algebraic morphism determined by 
$$
\C\to \C, \quad x\mapsto x^n,
$$
where $x$ is the coordinate of $\C$ as affine line. Consider the $S$-morphism 
$$\phi=\id\times [n]: X\to X,$$ and the pullback Higgs field $\phi^*\theta$ on $f$, which is the adjoint of the composite morphism
$$
f^*T_S=\phi^*f^*T_S\stackrel{\phi^*\theta}{\to} \phi^*T_{X/S}\stackrel{d\phi}{\to} T_{X/S}.
$$
Let $s$ be a local coordinate of $S$. Then $\phi^*\theta$ is given by 
$$
\partial_s \mapsto (nx^{n-1}\partial_s)e,
$$
where $\partial_s$ is understood as the linear local coordinate $\Omega_S$ along the fiber direction. When $n>1$, it is a nonlinear Higgs field on a vector bundle. 
\end{example}

\begin{example}\label{nonabelian Kodaira-Spencer}
Let $S$ be a smooth complex algebraic curve, and $h: X\to S$ be a polarized family of smooth projective curves of genus $\geq 2$. For $r\geq 1$, let $g: M_{Hig}=M^r_{Hig}(X/S)\to S$ be the coarse moduli space of rank $r$ stable Higgs bundles with vanishing Chern classes over $X/S$, which is a holomorphic fibration over $S$. Let $(\sE,\Theta)$ be the universal Higgs bundle over $N_{Hig}$, where $N_{Hig}$ appears in the following Cartesian diagram
 
\[
		\xymatrix{ N_{Hig}\ar[d]_{g'}\ar[r]^{h'} &M_{Hig}\ar[d]^{g} \\
			X\ar[r]^{h} & S.}
		\]
We write the universal Higgs field in the form $\Theta: T_{N_{Hig}/M_{Hig}}\to \sE nd(\sE)$. Let $\Theta^{end}: \sE nd(\sE)\to \sE nd(\sE)\otimes \Omega_{N_{Hig}/M_{Hig}}$ be the associated endormophism Higgs bundle to $(\sE,\Theta)$. It follows from the integrality of $\Theta$ that 
$$
\Theta^{end}\circ \Theta: T_{N_{Hig}/M_{Hig}}\to \sE nd(\sE)\otimes \Omega_{N_{Hig}/M_{Hig}}.
$$
is zero. Let $\Omega^*_{Hig}(\sE nd(\sE),\Theta^{end})$ be the Higgs complex of $(\sE nd(\sE),\Theta^{end})$ which is expressed as
$$
\sE nd(\sE)\stackrel{\Theta^{end}}{\longrightarrow} \sE nd(\sE)\otimes \Omega_{N_{Hig}/M_{Hig}} \stackrel{\Theta^{end}}{\longrightarrow}\cdots.
$$
Then we obtain a morphism of complexes $(T_{N_{Hig}/M_{Hig}},0)\to \Omega^*_{Hig}(\sE nd(\sE),\Theta^{end})$. As $T_{N_{Hig}/M_{Hig}}\cong g'^*T_{X/S}$, Grothendieck spectral sequence gives a natural morphism
$$
R^1h_*T_{X/S}\otimes g'_*\sO_{N_{Hig}}\cong R^1h_*g'_*g'^*T_{X/S}\to  g_*\R^1h'_*\Omega^*_{Hig}(\sE nd(\sE),\Theta^{end})\cong g_*T_{M_{Hig}/S}.
$$
So we have the natural morphism $\tau:  R^1h_*T_{X/S} \to g_*T_{M_{Hig}/S}$. Let $\rho_{KS}: T_{S}\to R^1h_*T_{X/S}$ be the Kodaira-Spencer morphism of $g$. Finally, we define a holomorphic Higgs field on $g$ by
$$
\theta_{KS}=\tau\circ \rho_{KS}: T_S\to g_*T_{M_{Hig}/S}.
$$
When $h$ is a family of hyperbolic curves, it can be shown that $\theta_{KS}=0$ if and only if $\rho_{KS}=0$, namely the family $h$ is isotrivial. 
\end{example}

\begin{definition}\label{almost higgs fields}
Notation as in Definition \ref{almost connection}. An almost Higgs field on $(\alpha,T_{rel})$ is a differentiable global section of the complex vector bundle $\alpha^*\Omega_S\otimes T_{rel}$.  
\end{definition}
Note that the space of almost connections on $(\alpha,T_{rel})$ is a torsor under the vector space formed by almost Higgs fields on $(\alpha,T_{rel})$. Also, by adjointness, a holomorphic Higgs field on a holomorphic fibration $f$ defines an $\sO_X$-linear morphism $f^*T_S\to T_{X/S}$, and hence a holomorphic global section of $f^*\Omega_S\otimes T_{X/S}$. Thus a holomorphic Higgs field is naturally an almost Higgs field.

\section{Chern connections and complex conjugations}
In this section, we shall generalize the notion of chern/metric connections and complex conjugations associated to hermitian metrics on complex vector bundles to complex fiber bundles.
\begin{definition}\label{hermitian metric}
Let $(\alpha,T_{rel})$ be a complex fiber bundle over $S$. A hermitian metric on it is a closed differentiable two-form $\omega$ on $Z$ whose restriction $\omega_s:=\omega|_{\alpha^{-1}(s)}$ to each fiber is of type $(1,1)$ and positive definite.   
\end{definition}
  
\begin{definition}
Let $\omega$ be a hermitian metric on a complex fiber bundle $(\alpha,T_{rel})$ as above. A metric connection on $(\alpha,T_{rel})$ is an almost connection $\partial: \alpha^*T_S\to \frac{T_{Z_{\C}}}{T_{rel}}$ such that its image $\partial(\alpha^*T_S)$ is orthogonal to $T_{rel}$ with respect to $\omega$. Given a $\bar \partial$-operator on $(\alpha,T_{rel})$ with the corresponding almost complex structure $T$ on $\alpha$, the chern connection $\partial_{\omega}$ is the \emph{unique} almost connection on $(\alpha,T)$ which is also a metric connection at the same time. 
\end{definition}
We would like to take a closer study of a chern connection when the $\bar \partial$-operator is integrable. So we are considering a holomorphic fibration
$f: X\to S$, and $\omega$ a hermitian metric on $f$. First, we look at $\partial_\omega$ locally. Let $x\in X$ and $s=f(x)\in S$. As in the proof of Proposition \ref{almost cmoplex structures in terms of d bar operators}, we take a set of local coordinates $\{s_1, \cdots, s_m,t_1,\cdots,t_r\}$ of $X$ around $x$, and a set of local coordinates $\{s_1,\cdots,s_m\}$ of $S$ around $s$ such that $f$ is given by the natural projection 
$$
(s_1,\cdots,s_m,t_1,\cdots,t_r)\mapsto (s_1,\cdots,s_m).
$$
By assumption, the matrix $A:=(\omega(\partial_{t_i},\partial_{t_j}))_{r\times r}$ is positive definite and hence non-degenerate in particular. Set $B:=(\omega(\partial_{s_i},\partial_{t_j}))_{m\times r}$. Then the association
$$
\partial_{s_i}\mapsto \partial_{s_i}+\sum_j c_{ij}\partial_{t_j},
$$
where $C=(c_{ij})_{m\times r}=B_{m\times r}A_{r\times r}^{-1}$, gives a local expression for $\partial_{\omega}$.  
\begin{example}\label{hermitian  metrics on vector bundles}
Let $f: V\to S$ be a holomorphic vector bundle over $S$, and $h$ a hermitian metric on $V$ in the usual sense. Take a local holomorphic basis $\{e_1,\cdots,e_r\}$ of $V$ over an open subset $U\subset S$ with a set of local coordinates $\{s_1,\cdots,s_m\}$ as above. Let $\{e_i^*\}_{1\leq i\leq r}$ be the dual basis to $\{e_i\}$. Then $f^{-1}(U)$ admits a set of local coordinates $\{s_1,\cdots,s_m,e_1^*,\cdots, e_r^*\}$. Set $h_{ij}=h(e_i,e_j)$ and $H=(h_{ij})$. For each $1\leq k\leq m$, set $\partial_{s_k}H=(\frac{\partial h_{ij}}{\partial s_k})$. Also set 
$$
b_{ij}=\sum_{1\leq k\leq r}e_k^*\frac{\partial h_{kj}}{\partial s_i},\quad 1\leq i\leq m, 1\leq j\leq r.
$$
Define a two-form $\omega_U$ over $U$ by 
\begin{eqnarray*}
\omega_U&=&\sqrt{-1}(\sum_{1\leq k,l\leq r}h_{kl}d e_k^*\wedge d\bar e_l^*  \\
&+&\sum_{1\leq i\leq m,1\leq j\leq r}b_{ij}ds_i\wedge d \bar e_j^* +\sum_{1\leq i\leq m,1\leq j\leq r}\bar b_{ij}de_j^*\wedge d \bar s_i\\
&+& \sum_{1 \le i, j \le m} \left( \sum_{1 \le k, l \le r} e_k^* \bar{e}_l^* \frac{\partial^2 h_{kl}}{\partial s_i \partial \bar{s}_j} \right) ds_i \wedge d\bar{s}_j). 
\end{eqnarray*}
It is straightforward to verify that $\{\omega_U\}_{U\subset S}$ glue together into a global two-form $\omega$. In fact \footnote{This insight is due to Nianzi Li.}, 
$$\omega = \sqrt{-1} \partial \bar{\partial}\sum_{k,l} h_{kl}(s) e_k^* \bar{e}_l^*.$$
Consequently, $d\omega=0$. Now we compare the chern connection associated to $\omega$ in the above sense with the chern connection associated to $h$ in the classical sense. It is a local check. The chern connection $\nabla^h$ has the connection one-form with respect to $\{e_i\}$ which reads
$$
\nabla^h(e_i)=\sum_{1\leq j\leq r, 1\leq k\leq m} \nabla^k_{ij} e_j\otimes ds_k,\quad \nabla^k_{ij}=(\partial_{s_k}H\cdot H^{-1})_{ij}. 
$$
So the associated transversal foliation to $\nabla^h$ over $U$ takes the form
$$
\partial_{s_i}\mapsto \partial_{s_i}+\sum_{1\leq j\leq r}(\sum_{1\leq k\leq r}\nabla^{i}_{kj}e^*_k)e_j.
$$
Here we regard $\{e_1,\cdots,e_r\}$ as a local holomorphic basis for $T_{V/S}$. On the other hand, by the local expression for chern connection as above, one finds immediately that the coefficient in front of $e_j$ is exactly the one in the above transversal foliation. Hence they are equal. 
\end{example}

\begin{example}\label{hermitian metrics on projective bundles}
Let $f: \P(V)\to S$ be the associated projective bundle. The hermitian metric $h$ induces a fiberwise Fubini-Study metric $\bar h$ on $\P(V)$. Let $\sO(1)$ be the tautological line bundle which is a holomorphic subbundle of $f^*V$. So it is equipped with the induced metric from $f^*h$ over $f^*V$. Let $\omega$ be the curvature form of this metric on $\sO(1)$. We claim that the chern connection associated to $\omega$ is the chern connection of $\bar h$ on the projective bundle $\P(V)$ in the classical sense. Again, we verify it locally. So we may assume $V$ is trivial, and we shall use notations from the previous example. We shall demonstrate the case $m=1, r=2$ only, since the verification for the general case is essentially the same. Then $\P(V)\cong S\times \P^{1}$ with $[e_1^*:e_2^*]$ the homogenous coordinate on $\P^{1}$. Set $x=\frac{e_1^*}{e_2^*}$. Thus $\{s,x\}$ forms a set of local coordinates of $S\times U$, where $U\subset \P^{1}$ is the principal open subset defined by $e_{2}^*\neq 0$. The chern connection associated to $\bar h$ is induced from the chern connection associated to $h$ on $V$. To deduce it, we make the change of coordinates $e_1^*\mapsto xe_2^*, e_2^*\mapsto e_2^*$. Then it is easy to get the expression of $a_{11}e_1^*e_1+a_{12}e_1^*e_2+a_{21}e_2^*e_1+a_{22}e_2^*e_2$ under the new coordinates, which is
$$
[-a_{12}x^2+(a_{11}-a_{22})x+a_{21}]\partial_{x}+(a_{12}xe_2^*+a_{22}e_2^*)e_2. 
$$
Therefore, we get the local expression of chern connection as follows
$$
\partial_s\mapsto \partial_s+[-a_{12}x^2+(a_{11}-a_{22})x+a_{21}]\partial_{x}, 
$$
where
\begin{eqnarray*}
a_{12}&=&(h_{11}\partial_sh_{12}-h_{12}\partial_sh_{11})(h_{11}h_{22}-h_{12}h_{21})^{-1},\\
a_{11}-a_{22}&=&(h_{22}\partial_{s}h_{11}-h_{11}\partial_sh_{22}+h_{12}\partial_sh_{21}-h_{21}\partial_{s}h_{12})(h_{11}h_{22}-h_{12}h_{21})^{-1},\\
a_{21}&=&(h_{22}\partial_sh_{21}-h_{21}\partial_sh_{22})(h_{11}h_{22}-h_{12}h_{21})^{-1}.   
\end{eqnarray*} 
On the other hand, set 
$$
f(s,x)=|xe_1+e_2|^2_{h}=|x|^2h_{11}+xh_{12}+\bar{x}h_{21}+h_{22}.
$$
Then the curvature form of the induced metric on $\sO(1)\subset f^*V$ (locally determined by $xe_1+e_2\in \{e_1,e_2\}$) is given by 
$$
\omega=\frac{\sqrt{-1}}{2}\partial\bar\partial \log f=(\omega_{ij}).
$$
Therefore, by the above local expression of the associated chern connection to $\omega$, we deduce that  
$$
\partial_s\mapsto \partial_s-\frac{\omega_{12}}{\omega_{11}}\partial_x,
$$
where 
\begin{eqnarray*}
\omega_{11}&=&f^{-2}(fh_{11}-|xh_{11}+h_{21}|^2)=f^{-2}(h_{11}h_{22}-h_{12}h_{21}),\\
\omega_{12}&=&f^{-2}[f(x\partial_sh_{11}+\partial_sh_{21})-\partial_sf(xh_{11}+h_{21})]=f^{-2}(x^2(h_{12}\partial_sh_{11}-h_{11}\partial_sh_{12})\\
           &+&x(h_{12}\partial_sh_{21}+h_{22}\partial_sh_{11}-h_{21}\partial_sh_{12}-h_{11}\partial_sh_{22})+h_{22}\partial_sh_{21}-h_{21}\partial_sh_{22}.
\end{eqnarray*} 
Now it is clear that these two expressions are the same.  The claim is proved. When $S$ is a compact Riemann surface, then GAGA implies that $f$ is algebraic. By Tsen's theorem, any projective bundle over a smooth projective curve is the projectification of a vector bundle. However, it is not the case when $\dim S\geq 2$. But when it is equipped with a relative ample line bundle $L$, it is still meaningful to study the chern connection in the new sense, which is associated to the curvature form representing $c_1(L)$.  
\end{example}
Next, we want to discuss a condition on an almost connection on a holomorphic fibration so that we can talk about a curvature tensor which is the obstruction for it to be a holomorphic one. 
\begin{lemma}\label{curvature of type (1,1)}
Let $f: X\to S$ be a holomorphic fibration and $\partial$ an almost connection on $f$, viz. on $(\alpha,T_X)$. Then the followings hold true:
\begin{itemize}
    \item [(1)] Let the association $\partial_{s_i}\mapsto \partial_{s_i}+\sum_jc_{ij}\partial_{t_j}$ be the local expression of $\partial$ with respect to a set of holomorphic coordinates $\{s_i, t_j\}_{1\leq i\leq m, 1\leq j\leq r}$ of $f$ as above. Then $\frac{\partial c_{ij}}{\bar \partial_{t_k}}=0$ for any $i,j,k$ if and only if the following property
    $$
    [\bar T_{X/S},\textrm{im}(\partial)]\subset \bar T_{X/S}
    $$
    holds. When one of these properties of $\partial$ is satisfied, we say that it is holomorphic along vertical direction. 
    \item [(2)] For an almost connection $\partial$ which is holomorphic along vertical connection, there is a well-defined differentiable bundle map $F^{1,1}: T_S\otimes \bar T_S\to f_*T_{X/S}$, which is locally given by
    $$
    \partial_{s_i}\otimes \bar \partial_{s_j} \mapsto \overline{[\partial(\partial_{s_i}),\bar \partial_{s_j}]}, 
    $$
    such that it is zero if and only if $\partial$ is holomorphic.
\end{itemize}
\end{lemma}
\begin{proof}
To show the latter property in (1), it suffices to show it for a set of basis elements in $\bar T_{X/S}$. So we may take $\{\bar \partial_{t_k}\}_{1\leq k\leq r}$. As 
$$
[\bar \partial_{t_k},\partial_{s_i}+\sum_jc_{ij}\partial_{t_j}]=\sum_j\frac{\partial c_{ij}}{\bar \partial_{t_k}}\partial_{t_j},
$$
it is clear that the bracket belongs $\bar T_{X/S}$ for any $i,k$ if and only if the former property in (1) holds. This concludes (1). To show (2), we first provide a change of variable formula for $C=(c_{ij})$. As in the proof of Proposition \ref{almost cmoplex structures in terms of d bar operators}, we let $\{s'_1,\cdots,s'_m\}$ be another set of local coordinates of $f(x)\in S$, and $\{s'_1,\cdots,s'_m,t'_1,\cdots,t'_r\}$ another set of local coordinates of $x\in X$.  Set $S=(s_{ij})_{m\times m}$ with $s_{ij}=\frac{\partial s'_j}{\partial s_i}$ and $T=(t_{ij})_{r\times r}$ with $t_{ij}=\frac{\partial t'_j}{\partial t_i}$. Then $S,T$ are matrices of holomorphic functions in $\{s_i,t_j\}$s. Let $C'=(c'_{ij})_{m\times r}$ be the representation matrix of $\partial_{\omega}$ with respect to $\{s'_i,t'_j\}$s. By the chain rule, we obtain the equality of matrices 
$$
C'=S^{-1}\cdot C\cdot T.
$$
For (2), we need to show the association
$$
\partial s_i\otimes \bar \partial s_j\mapsto -\sum_k\frac{\partial c_{ij}}{\bar \partial s_j}\partial t_k
$$
is invariant under the above change of variables. Because the almost connection is holomorphic along vertical direction by condition, it follows from the chain rule that
$$
\frac{\partial c'_{ij}}{\bar \partial s'_k}=\sum_l\frac{\partial c'_{ij}}{\bar \partial s_l}\frac{\bar \partial s_{l}}{\bar \partial s'_k}.
$$
Now the invariance follows immediately from the change of variable formula for $C'$ given as above. Therefore we have obtained a curvature tensor of type $(1,1)$ in base with coefficient in $f_*T_{X/S}$. By definition, $\partial$ is holomorphic when $c_{ij}$s are holomorphic functions in $\{s_i,t_j\}$s, that is $\frac{\partial c_{ij}}{\bar \partial s_k}=\frac{\partial c_{ij}}{\bar \partial t_l}=0$ for all $i,j,k,l$. The remaining statement follows just because of the condition on $\partial_{\omega}$ and the above local expression of the curvature tensor.    
\end{proof}

Finally, we shall define the complex conjugation of an almost Higgs field on $f$. First, let us notice that, for each $s\in S$, the identity component of the automorphism group $\Aut^0(X_s)$ has a natural structure of connected complex Lie group, which is however of infinite dimensional in general. That is, its Lie algebra $\mathfrak{aut}(X_s)=H^0(X_s,T_{X_s})$, the complex linear space of global holomorphic tangent vector fields, is generally of infinite dimensional dimension. Let $\theta: T_S\to f_*T_{X/S}$ be an almost Higgs field on $f$. Then for $s\in S$, $\theta$ defines
a natural $\C$-linear map
$$
\theta(s): T_S(s)\to f_*T_{X/S}(s)\cong \mathfrak{aut}(X_s),
$$
where $T_S(s)$ is of constant finite dimension.
\begin{definition}
A hermitian metric $\omega$ on $f$ is said to be auto-finite, if for each $s\in S$, the real Lie subgroup $\Aut(X_s,\omega_s)$ of holomorphic isometries is of finite dimension which is independent of $s$. For a auto-finite hermitian metric $\omega$, it is said to be $\theta$-adapted, for each $s\in S$, $\textrm{im}(\theta_s)$ is contained in the complexification $\mathfrak{aut}(X_s,\omega_s)_{\C}\subset \mathfrak{aut}(X_s)$.    
\end{definition}
\begin{example}\label{automorphism groups}
Let $f: V\to S$ be a holomorphic vector bundle and $\omega$ a hermitian metric on $f$ as given in Example \ref{hermitian  metrics on vector bundles}. Then $\Aut^0(X_s)\cong \Aut^0(\C^r)$ is infinite dimensional when $r>1$. However, $\Aut^0(X_s,\omega_s)$ is isomorphic to the extension of $U(r)$ by $\C^r$, and hence $\omega$ is auto-finite. When $f$ is a fibration of compact complex manifolds, the automorphism groups of its fibers are of finite dimension by a result Douady. In such a case, any hermitian metric on $f$ is auto-finite. 
\end{example}
 
Note that when $\omega$ is auto-finite, $\Aut^0(X/S,\omega):=\cup_{s\in S}\Aut^0(X_s,\omega_s)$ is a principal $K$-bundle over $S$, where $K$ is a connected compact Lie group. Here is a useful observation: When $\omega$ is further $\theta$-adapted for a holomorphic Higgs field $\theta$, we actually obtain a principal $G$-Higgs bundle over $S$, where $G$ is the complexification of $K$.  
\begin{definition}
Let $\theta$ be an almost Higgs field on $f$, and $\omega$ a auto-finite and $\theta$-adapted hermitian metric on $f$. The complex conjugation $\bar \theta_{\omega}$ of $\theta$ with respect to $\omega$ is the differentiable section of $\bar \Omega_S\otimes f_*T_{X/S}$ whose value at $s\in S$ is defined by sending $\bar \partial_s$ to the minus of the complex conjugation of $\theta(s)(\partial_s)$ with respect to the real form $\mathfrak{aut}(X_s,\omega_s)$ for any anti-holomorphic tangent vector $\bar \partial_s$ at $s$.
\end{definition}

\begin{example}\label{complex conjugation for vector bundles}
Let $(V,\theta)$ be a holomorphic Higgs bundle over $S$. Let $\omega$ be the hermitian metric on $f: V\to S$ as given in Example \ref{hermitian  metrics on vector bundles}. For $\theta$ being linear, $\omega$ is $\theta$-adapted. Let us compute the complex conjugation of $\theta$ with respect to $\omega$ explicitly. Assume $S$ to be local with a set of local coordinates $\{s_1,\cdots,s_m\}$, and $V$ is trivial with a set of holomorphic basis $\{e_1,\cdots,e_r\}$. Let 
$$
\theta=\sum \theta_k\otimes ds_k, \quad \theta_k=\sum_{i,j} \theta^k_{ij}(s)e_i^*\otimes e_j.
$$
Let $(\bar\theta^k_{ij})$ be the adjoint of the matrix $(\theta^k_{ij})$ with respect to the metric $h$. The associated Higgs field on $V\to S$ as holomorphic fiber bundle is given by
$$
\partial_{s_k}\mapsto \sum_j(\sum_{i}\theta^k_{ij}e_i^*)e_j,
$$
where we view $\{e_j\}_{1\leq j\leq r}$ as a holomorphic basis of $T_{V/S}$ and $\{e_i^*\}_{1\leq i\leq r}$ as elements of $\sO_V$. If no confusion will arise, we denote it again by $\theta$. Claim that the complex conjugation $ \bar\theta_{\omega}$ of $\theta$ with respect to $\omega$ is given by
$$
\bar \partial_{s_k}\mapsto \sum_{j}(\sum_i\bar\theta^k_{ij}e_i^*)e_j.
$$
This is simple linear algebra: Let $H=(h_{ij})$ for $h_{ij}=h(e_i,e_j)$, and let 
$$
\Theta^k=(\theta^k_{ij}),\ \bar \Theta^k=(\bar\theta^k_{ij}).$$
Then $\bar \Theta^k=Ad_H(\Theta^{k*})$. On the other hand, via the chosen basis, the natural inclusion $\mathfrak{aut}(V_s,\omega_s)\subset \mathfrak{aut}(V_s,\omega_s)_{\C}$ is identified with 
$$
Ad_{H^{\frac{1}{2}}}(\mathfrak{u}_r)\subset \mathfrak{gl}_r(\C).
$$
Therefore, by decomposing   
$$
\Theta^k=\frac{\Theta^k-Ad_H(\Theta^{k*})}{2}+\frac{\Theta^k+Ad_H(\Theta^{k*})}{2}
$$
into the sum of its real and imaginary parts, we immediately obtain the complex conjugation with respect to $\omega$ as follows:
$$
\bar \Theta^k=\frac{\Theta^k-Ad_H(\Theta^{k*})}{2}-\frac{\Theta^k+Ad_H(\Theta^{k*})}{2}=Ad_H(-\Theta^{k*}).
$$
The claim is proved by noticing the minus in the definition of complex conjugation given above.
\end{example} 
\section{Simpson mechanism}
In this section, we shall reformulate Simpson's construction in \cite[Construction \S1]{Si1} in a way that generalizes to the fiber bundle setting. Our setup is as follows: Let $\alpha: Z \to S$ be a differentiable fiber bundle over $S$, $T_A$ and $T_B$ are two integrable complex structures on $T_{Z_{\R}/S}$, and an automorphism of complex vector bundle $\beta: \frac{T_{Z_{\C}}}{\bar T_A} \to \frac{T_{Z_{\C}}}{\bar T_B}$ such that $\beta$ restricts to an automorphism $\beta: T_A\cong T_B$ and induces the identity map on $\alpha^*T_{S_{\C}}$. When $T_A=T_B$, a simple choice of $\beta$ is the identity map. \\

{\itshape From Higgs fields to almost connections:} Let $g: Y\to S$ be a holomorphic fibration, whose underlying complex fiber structure is $(\alpha,T_B)$, which is equipped with a holomorphic Higgs field $\theta$ \footnote{It seems reasonable to call the pair $(g,\theta)$ again by a Higgs bundle over $S$. When $g$ is a holomorphic vector bundle and $\theta$ is a \emph{linear} Higgs field on $g$, the pair $(g,\theta)$ is said to be a \emph{linear} Higgs bundle. Likewise, one may call a pair $(f,\nabla)$ simply by a flat bundle over $S$, where $f: X\to S$ is a holomorphic fibration equipped with a transversal foliation, and a linear flat bundle when it arises from the vector bundle setting.}. Let $\bar \partial_g$ be the complex structure of $g$, regarded as a $\bar \partial$-operator on $g$, namely
$$
\bar \partial_g: g^*\bar T_{S}\to \frac{T_{Y_{\C}}}{\bar T_{Y/S}}.
$$
Let $\omega$ be a $\theta$-adapted hermitian metric on $g$, and let $\bar \theta_{\omega}$ be the complex conjugation of $\theta$ with respect to $\omega$, which is a differentiable bundle morphism
$$
\bar \theta_{\omega}: \bar T_S\to g_*T_{Y/S}. 
$$
By abuse of notation, the following composite map  
$$
g^*\bar T_S\stackrel{g^*\bar \theta_{\omega}}{\longrightarrow} g^*g_*T_{Y/S}\to T_{Y/S}\to \frac{T_{Y_{\C}}}{\bar T_{Y/S}}
$$
is still denoted by $\bar \theta_{\omega}$. Now we define a differentiable bundle morphism $\alpha^*\bar T_S\to \frac{T_{Z_{\C}}}{\bar T_A}$ by the formula
\begin{equation}
\bar\partial_{\omega}:=\beta^{-1}\circ (\bar\partial_{g}+\bar\theta_{\omega}).   
\end{equation}

Since its composite with the projection $\frac{T_{Z_{\C}}}{\bar T_A}\to \alpha^*\bar T_S$ is the identity, it defines an almost complex structure on $(\alpha,T_A)$. By Example \ref{complex conjugation for vector bundles}, the above formula for a linear Higgs bundle and $\beta=id$ is exactly the reconstruction formula for the complex structure of the corresponding linear flat bundle. Set $T_A\subset T_{\omega}$ to be the almost complex structure on $\alpha$ corresponding to $\bar \partial_{\omega}$. Next, we define an almost connection on $(\alpha,T_{\omega})$ by the formula 
\begin{equation}
\nabla_\omega:=\partial_\omega+2\beta^{-1}\circ \theta,
\end{equation}
where $\partial_{\omega}$ is the Chern connection associated to $\bar \partial_{\omega}$ with respect to $\omega$. When $(g,\theta)$ is a linear Higgs bundle, $\beta=id$ and $\omega$ is a hermitian metric on $g$ in Example \ref{hermitian  metrics on vector bundles}, the above formula is the reconstruction formula for connection.   \\

{\itshape From connections to almost Higgs fields:} Let $f: X\to S$ be a holomorphic fibration over $S$, whose underlying complex fiber bundle structure is $(\alpha,T_A)$, which is equipped with a holomorphic connection $\nabla$ on $f$. For a hermitian metric $\omega$ on $f$, we define an almost Higgs field on $(\alpha,T_B)$ by the formula
\begin{equation}
\theta_{\omega}:=\frac{1}{2}\beta\circ (\nabla-\partial_{\omega}).    
\end{equation}
In order to proceed further, we need to choose $\omega$ so well that $\omega$ is $\theta_{\omega}$-adapted. This will be automatically satisfied in two cases: The linear case, and the case when $\Aut^0(X/S)$ is a principal $G$-bundle and $\Aut^0(X/S,\omega)$ is principal $K$-bundle, where $K$ is a real form of $G$. Under this assumption, the complex conjugation $\bar \theta_{\omega}$ of $\theta_{\omega}$ with respect to $\omega$ is well-defined. Next we define an almost complex structure on $(\alpha,T_B)$ by the formula
\begin{equation}
\bar\partial_{\omega}:=\beta\circ\bar\partial_f-\bar \theta_{\omega}.   
\end{equation}
Again, when $(f,\nabla)$ is a linear flat bundle, $\beta=id$ and $\omega$ is a hermitian metric on $f$ in Example \ref{hermitian  metrics on vector bundles}, the above formulas are the reconstruction formulas for the corresponding linear Higgs bundle. 
\begin{remark}
It is manifest that the above two constructions are converse to each other. Note that \emph{not} every operator appearing in Simpson's construction has been generalized in our setting: The operator $\delta^{''}$ (see \cite[Construction \S1]{Si1}) is the unique operator of type $(0,1)$ such that $\delta^{''}+d'$ (where $d'$ is the holomorphic flat connection $\nabla$ in our terminology) preserves the hermitian metric. We may also do so if we include anti-holomorphic objects into the study. Consequently, we have to replace it by $\bar \partial-\bar\theta$, and hence obtain the above construction from connections to almost Higgs fields. The construction from Higgs fields to almost connections is literally the converse process. 
\end{remark}

\section{Nonlinear harmonic bundles}
In the linear case, a flat bundle $(V,\nabla)$ and its corresponding Higgs bundle $(E,\theta)$ share the same underlying \emph{complex} vector bundle structure over $S$: Let $f: V\to S$ and $g: E\to S$ be the associated geometric vector bundle. Then $T_{V/S}$ and $T_{E/S}$ are canonically identified as complex vector subbundle of $T_{Z_{\R}/S}\otimes \C$, where $Z \to S$ is the underlying differentiable fiber bundle of $f$ (and also $g$). This means that, in the linear case, there is a natural identification of the complex structures on the relative tangent bundle imposed by two corresponding objects. More importantly, each direction of the above construction in the linear case provides a curvature tensor, which is the \emph{only} obstruction to obtain a genuine mirror structure (\cite[Curvature \S1]{Si1}). In order to achieve this property for a general case, it is natural for us to impose a stronger assumption on $\beta$ in the previous setting for Simpson mechanism. Namely, we assume a real automorphism $\beta: T_{Z_{\C}}\to T_{Z_{\C}}$ so that $\beta(T_A)=T_B$, induces the identity on $\alpha^*T_{S_{\C}}$, and preserves the local lifting condition: For a local lifting $\tilde{\bar{\partial}}: \alpha^*\bar T_S\to T_{Z_{\C}}$ of a $\bar \partial$-operator on $(\alpha, T_A)$, or a local lifting $\tilde \partial: \alpha^*T_S\to T_{Z_{\C}}$ of an almost connection $\partial$ on $(\alpha,T_A)$, the local lifting condition holds
$$
[\bar T_B, \beta(im(\tilde{\bar{\partial}})]\subset \bar T_B;\quad [\bar T_B, \beta(im(\tilde \partial))]\subset \bar T_B.
$$

Now let $g: Y\to S$ be a holomorphic fibration whose underlying complex fiber bundle structure is $(\alpha,T_B)$. It is also equipped with a holomorphic Higgs field $\theta$. Let $\omega$ be a hermitian metric on $g$ which is preserved by $\beta$, and such that the associated chern connection $\partial_{\omega}$ is holomorphic along vertical direction (see Lemma \ref{curvature of type (1,1)}) and is $\theta$-adapted. Consider the following differentiable transversal foliation
$$
D_{\omega}:=\nabla_{\omega}+\bar \partial_{\omega}: \alpha^*T_{S_{\C}}\to \frac{T_{Z_{\C}}}{\bar T_{A}}.
$$
Here we have embedded $T_{\omega}$ into $\frac{T_{Z_{\C}}}{\bar T_A}$. Since the Lie algebra structure on $T_{Z_{\C}}$ does not descend to a Lie algebra structure on the quotient $\frac{T_{Z_{\C}}}{\bar T_{A}}$ (that requires $\bar T_A$ be an ideal which is not the case in general), we shall not immediately form a curvature tensor as in the linear case. Instead, we shall consider separately three curvature tensors according to their types. \\

{\itshape Type $(0,2)$}: Consider the almost complex structure $\bar \partial_{\omega}$ in the above construction. We shall verify the local lifting condition in Lemma \ref{integrability of almost complex structure}. Since $\beta$ preserves the local lifting condition, it is equivalent to show the condition for $\bar \partial_g+\bar \theta_{\omega}$. As $\bar \partial_g$ is integrable, it suffices to show $[\bar T_B,\bar \theta_{\omega}]\subset \bar T_B$. As $\theta$ is holomorphic, it is holomorphic along vertical direction. But since the complex conjugation does not change the holomorphicity along vertical direction, $\bar \theta_{\omega}$ is also holomorphic along vertical direction. So the local lifting condition holds for $\bar \partial_{\omega}$. Then Lemma \ref{integrability of almost complex structure} gives us a curvature tensor $F^{0,2}$ which is a differentiable global section of $\bar \Omega^2_{S}\otimes \alpha_*T_{\omega}$ such that $F^{0,2}=0$ if and only if $\bar \partial_{\omega}$ is integrable. \\

{\itshape Type $(1,1)$}: Suppose $F^{0,2}=0$. We proceed to verify the condition in Lemma \ref{curvature of type (1,1)} (1) for $\nabla_{\omega}$, that is 
$$
[\bar T_A, \textrm{im}(\partial_{\omega}+2\beta^{-1}\circ\theta)]\subset \bar T_A. 
$$
As $\beta$ is local lifting condition preserving and $\theta$ is holomorphic, it is equivalent to verify the condition
$$
[\bar T_B, \textrm{im}(\beta\circ \partial_{\omega})]\subset \bar T_B. 
$$
Since $\beta$ preserves $\omega$, $\beta$ transforms the chern connection associated to $\bar \partial_{\omega}$ to the chern connection associated to $\bar \partial_g+\bar \theta_{\omega}$. To get the condition as requested, we shall compare the chern connections associated to $\bar \partial_g$ and $\bar \partial_g+\bar \theta_{\omega}$ respectively. We write the former by $\partial$ and the latter by $\partial^{\theta}$ for a temporary use. Take a set of holomorphic local coordinates $\{s_i,t_j\}$ of $g$. A holomorphic basis for $T_Y$ is given by $\{\partial_{s_i},\partial_{t_j}\}$, while a holomorphic basis for $\beta\circ T_{\omega}$, the holomorphic tangent bundle for the new complex structure $\bar \partial_g+\bar \theta_{\omega}$, is given by $\{\partial_{s_i}+\sum_{k}\overline{\bar{\theta}_{ik}}\bar \partial_{t_k},\partial_{t_j}\}$. Here $\bar \theta_{\omega}$ is given by the association $\partial_{s_i}\mapsto \sum_k\bar \theta_{ik}\partial_{t_k}$ and $\overline{\bar \theta_{ik}}$ is the complex conjugation of the function $\bar \theta_{ik}$. We may write $$
\partial^{\theta}(\partial_{s_i})=\partial(\partial_{s_i})+\sum_ja_{ij}\partial_{t_j}+\sum_jb_{ij}\bar \partial_{t_j}.$$ 
The coefficients $a_{ij}$s and $b_{ij}$s are uniquely determined by two conditions: $$(\partial^{\theta}(\partial_{s_i}),\partial_{t_j})_{\omega}=0$$ for any $j$, and $\partial^{\theta}(\partial_{s_i})\in T_{\omega}$. Using the fact that $(\bar \partial_{t_i},\partial_{t_j})_{\omega}=0$ and the matrix $((\partial_{t_i},\partial_{t_j})_{\omega})$ is invertible, we easily obtain that 
$$
a_{ij}=0,\quad b_{ij}=\overline{\bar \theta_{ij}}. 
$$
But then since $\partial$ is holomorphic along vertical direction by condition, it follows that
$$
[\bar \partial_{t_k}, \partial(\partial_{s_i})+\sum_j\overline{\bar \theta_{ij}}\bar \partial_{t_j}]\in \bar T_B.
$$
So we are done. By Lemma \ref{curvature of type (1,1)}, we obtain the second curvature tensor $F^{1,1}$ associated to $\nabla_{\omega}$, which is a differentiable global section of $\Omega_S\otimes \bar \Omega_{S}\otimes \alpha_*T_{A}$, and which vanishes if and only if $\nabla_\omega$ is a holomorphic connection with respect to the complex structure $\bar \partial_{\omega}$. Because of the assumption on $\beta$, we may also interpret $F^{1,1}$ as a global section of $\Omega_S\otimes \bar \Omega_{S}\otimes g_*T_{Y/S}$.\\

{\itshape Type $(2,0)$}: Suppose $F^{1,1}=0$. Let $f: X\to S$ be the holomorphic fibration corresponding to $\bar \partial_{\omega}$. It follows from Definition \ref{holomorphic connection} that the holomorphic connection $\nabla_{\omega}$ on $f$ is integrable if and only if the curvature tensor $F^{2,0}$ vanishes, which is a holomorphic global section of $\Omega^2_S\otimes f_*T_{X/S}$. Like above, we may reinterpret $F^{2,0}$ as a global section of $\Omega_S^2\otimes g_*T_{Y/S}$, which is however not necessarily holomorphic.\\

We may summarize the above discussions into a clean statement. First we make a definition.
\begin{definition}
Let $(\alpha,T_{rel})$ be a complex fiber bundle. Let $D^{0,1}$ be a $\bar \partial$-operator on $(\alpha,T_{rel})$ and $D^{1,0}$ be an almost connection on $(\alpha,T)$, where $T\subset T_{Z_{\C}}$ corresponds to $D^{0,1}$. Form $D=D^{1,0}+D^{0,1}: \alpha^{*}T_{S_{\C}} \to \frac{T_{Z_{\C}}}{\bar T_{rel}}$, which is a splitting of the natural projection $\frac{T_{Z_{\C}}}{\bar T_{rel}}\to \alpha^{*}T_{S_{\C}}$. We say $D$ is holomorphic along vertical direction if $D^{0,1}$ admits the local lifting condition in Lemma \ref{integrability of almost complex structure} and $D^{1,0}$ satisfies $[\bar T_{rel},\textrm{im}(D^{1,0})]\subset \bar T_{rel}$ as in Lemma \ref{curvature of type (1,1)}. 
\end{definition}
It is suggestive to write the above condition by 
$$
[\bar T_{rel},\textrm{im}(D)]\subset \bar T_{rel}.
$$
 
\begin{proposition}\label{vanishing of F}
Notations as above. Then $D_{\omega}$ satisfies the holomorphic along vertical direction condition. There exists a curvature tensor $F_{\omega}$ attached to $D_{\omega}$, which is a differentiable global section of the complex vector bundle $\Omega^2_{S_{\C}}\otimes g_*T_{X/S}$ over $S$, such that the pair $(\bar \partial_{\omega},\nabla_{\omega})$ constructed as above defines a holomorphic flat bundle $(f,\nabla)$ with $(\alpha,T_A)$ as the underlying complex fiber bundle, if and only if $F_{\omega}=0$.
\end{proposition}
\begin{proof}
Except that $F^{0,2}$ actually belongs to $\Gamma(S,\bar \Omega_S^2\otimes g_*T_{Y/S})$, all have been proved in the previous paragraphs. This holds because of the special form of the $\bar \partial$-operator $\bar \partial_g+\bar \theta_{\omega}$. As $\beta$ is a Lie algebra automorphism, we may reinterpret $F^{0,2}$ as a section of $\bar \Omega^2_S\otimes g_*(\beta\circ T_{\omega})$. Retaining the notation in the paragraph on type $(1,1)$-curvature tensor, $\beta\circ\bar T_{\omega}$ is locally generated by $\{\bar \partial_{t_j},\bar \partial_{s_i}+\sum_k\bar\theta_{ik}\partial_{t_k}\}$. The curvature tensor $F^{0,2}$ results from taking the composite map
$$
\wedge^2 (\beta\circ\bar T_{\omega})\stackrel{[\ ,\ ]}{\to}T_{Z_{\C}}\to \frac{T_{Z_{\C}}}{\beta\circ \bar T_{\omega}}=\beta\circ T_{\omega}. 
$$
But since $\bar T_{B}$ is closed under Lie bracket and $\bar \theta_{\omega}$ is holomorphic along vertical direction, it follows that in the above bracket, only the terms $\{\frac{\partial \bar\theta_{ik}}{\bar \partial_{s_j}}\partial_{t_k}\}$s remain. Clearly, they all lie in $T_{B}\subset T_{\omega}$. We may conclude the proof by setting $F_{\omega}=F^{0,2}+F^{1,1}+F^{2,0}$. 
\end{proof}
Now we reverse the direction. Namely, we start with a holomorphic flat bundle $(f,\nabla)$ with underlying complex fiber bundle structure $(\alpha,T_A)$. Let $\omega$ be a hermitian metric on $f$ which is preserved by $\beta$, whose associated chern connection is horizontal along vertical direction, and which is $\theta_{\omega}$-adapted. Form  
$$
D''_{\omega}:=\bar\partial_{\omega}+\theta_{\omega}: \alpha^*T_{S_{\C}}\to \frac{T_{Z_{\C}}}{\bar T_{B}}.
$$
Argued as before, it follows from our assumption that $\theta_{\omega}$ is holomorphic along vertical direction. So is its complex conjugation $\bar \theta_{\omega}$. Therefore $\bar \partial_{\omega}$ is holomorphic along vertical direction. Lemma \ref{integrability of almost complex structure} provides us a curvature tensor $G^{0,2}\in \Gamma(S,\bar \Omega_S^2\otimes \alpha_*T_{\omega})$, where $T_{\omega}$ corresponds to $\bar \partial_{\omega}$. Using exactly the same reasoning as in the proof of Proposition \ref{vanishing of F}, one may interpret $G^{0,2}$ as an element in $\Gamma(S,\bar \Omega_S^2\otimes f_*T_{X/S})$, whose vanishing is equivalent to the integrability of $\bar \partial_{\omega}$. To obtain the curvature tensor $G^{1,1}\in \Gamma(S,\Omega_S\otimes \bar \Omega_S\otimes f_*T_{X/S})$ of type $(1,1)$, whose vanishing is equivalent to the holomorphicity of $\theta_{\omega}$, one repeats the argument in Lemma \ref{curvature of type (1,1)} and the construction of $F^{1,1}$ given above. Finally, the integrality of $\theta_{\omega}$ is measured by a curvature tensor $G^{2,0}\in \Gamma(S,\Omega^2_S\otimes f_*T_{X/S})$ given by the association $a\wedge b\mapsto \beta^{-1}([\theta_{\omega}(a),\theta_{\omega}(b)])$. 
\begin{proposition}\label{vanishing of G}
Notations as above. Then there exists a curvature tensor $G_{\omega}\in \Gamma(S,\Omega^2_{S_{\C}}\otimes f_*T_{X/S})$ associated to $D_{\omega}^{''}$, which vanishes if and only if the pair $(\bar \partial_{\omega},\theta_{\omega})$ defines a holomorphic Higgs bundle $(g,\theta)$ over $S$, whose underlying complex fiber structure is given by $(\alpha,T_{B})$.   
\end{proposition}
\begin{proof}
Set $G_{\omega}=G^{0,2}+G^{1,1}+G^{2,0}$. 
\end{proof}
\begin{definition}
Let $S$ be a complex manifold. For a holomorphic flat bundle $(f,\nabla)$ (resp. a holomorphic Higgs bundle $(g,\theta)$) over $S$, a hermitian metric $\omega$ on $f$ (resp. $g$) is said to be allowable if it is preserved by $\beta$, the associated chern connection is holomorphic along vertical direction, and it is $\theta_{\omega}$ (resp. $\theta$)-adapted. A harmonic metric on $(f,\nabla)$ (resp. $(g,\theta)$) is an allowable hermitian metric $\omega$ on $f$ (resp. $g$) such that $G_{\omega}=0$ (resp. $F_{\omega}=0$).  
\end{definition}
The notion of a harmonic bundle a la Simpson can now have a much broader meaning. 
\begin{definition}\label{nonlinear harmonic bundle}
Let $S$ be a complex manifold. A nonlinear harmonic bundle over $S$ is a quadruple $(\alpha,T_A,T_B,\beta)$, where $\alpha: Z\to S$ is a differentiable fiber bundle over $S$, $T_A$ and $T_B$ are two integrable complex structures on $T_{Z_{\R}/S}$, $\beta: T_{Z_{\R}}\to T_{Z_{\R}}$ is a local lifting condition preserving bundle automorphism rendering the following commutative diagram
\begin{equation*}
		\begin{gathered}
		\xymatrix{
			0\ar[r]&T_{A} \ar[r] \ar[d]^{\beta} & \frac{T_{Z_{\C}}}{\bar T_A} \ar[d]^{\beta}\ar[r]&\alpha^*T_{S_{\C}}\ar[r]\ar[d]^{id}&0\\
			0\ar[r]&T_{B}\ar[r] & \frac{T_{Z_{\C}}}{\bar T_A}\ar[r]&\alpha^*T_{S_{\C}}\ar[r]&0,
		}
		\end{gathered}  
		\end{equation*}
provided with structures of holomorphhic flat bundle whose underlying complex fiber bundle is $(\alpha,T_A)$ and holomorphic Higgs bundle whose underlying complex fiber bundle is $(\alpha,T_B)$,  which are related by a harmonic metric. A choice of such metric is not part of the data. 
\end{definition}
\section{Relative nonabelian Hodge moduli spaces}
Let $S$ be a smooth complex curve and $X\to S$ be a polarized smooth family of smooth projective curves of genus $\geq 2$. Let $f: M^{s}_{dR}(X/S,r)\to S$ be the relative de Rham moduli space, parametrizing rank $r$ irreducible holomorphic connections on $X_s$, and $g: M^s_{Hig}(X/S,r)\to S$ be the relative Higgs moduli space of rank $r$ stable Higgs bundles of degree zero. The Gauss-Manin connection $\nabla_{GM}$, which is constructed in \cite[\S8]{Si2} as an algebraic connection, comes from the isomonodromy deformation of a flat connection. On the other hand, Example \ref{nonabelian Kodaira-Spencer} equips $g$ with a holomorphic Higgs field. In \cite{Si3}, the Hodge filtration on $f$ is defined, which we shall denote by $F_{hod}$. The grading $Gr_{F_{hod}}(f,\nabla_{GM})$ is defined to be the fiber of $M_{hod}\to \A^1\times S$ over $\{0\}\times S$, equipped with the residual action of the $\G_m$-equivariant extension of $\nabla_{GM}$ over $M_{hod}$, which covers the standard $\G_m$-action on $\A^1\times S$ (here the trivial action of the factor $S$ is taken). Then it has been shown that (see \cite[Theorem 1.2]{FS})
$$
Gr_{F_{hod}}(f,\nabla_{GM})=(g,\theta_{KS}).
$$
By \cite[Theorem 7.18]{Si2}, the nonabelian Hodge correspondence gives a canonical homeomorphism $\alpha_S: M^{s}_{dR}(X/S,r)\cong M^{s}_{Hig}(X/S,r)$. It is clear that $\alpha_S$ gives an isomorphism between $f$ and $g$ as \emph{topological} fiber bundle over $S$. 
\begin{conjecture}\label{identification of differentiable structrues}
Notation as above. Then $\alpha_S$ is an isomorphism of differentiable fiber bundles.     
\end{conjecture} 
Recall that over each fiber $Y_s$, Hitchin showed that the $L^2$ metric on the tangent spaces induced by the harmonic metrics (unique up to constant) is hyperk\"ahler. 
\begin{conjecture}\label{Hitchin metric glues}
There is a hermitian metric $\omega_H$ on $g$, unique up to adding the pullback of a closed two-form on $S$, which is $\theta_{KS}$-adapted, and its restriction to each fiber $Y_s$ is a Hitchin's metric.    
\end{conjecture}
We shall call the conjectural hermitian metric $\omega_H$ the Hitchin metric on $g$. The final conjecture is the following.
\begin{conjecture}\label{nonabelian Hodge moduli conjecture}
Assume the truth of Conjecture \ref{identification of differentiable structrues}, that is, one identifies the differentiable fiber bundle structures of $f$ and $g$ via $\alpha_S$. The resulting differentiable fiber bundle is denoted by $\alpha: Z\to S$. Assume also the existence of a Hitchin metric $\omega_H$ in Conjecture \ref{Hitchin metric glues}. Then it is a harmonic metric on $\alpha$ with respect to a suitably defined $\beta: T_{Z_{\R}}\cong T_{Z_{\R}}$, and therefore $(f,\nabla^{GM};g,\theta^{KS})$ becomes a nonlinear harmonic bundle over $S$.
\end{conjecture}
In the following, we are going to prove the conjecture in the simplest case, viz. $r=1$. 
\begin{proposition}\label{rank one moduli}
Notations as above. Conjecture \ref{nonabelian Hodge moduli conjecture} holds for $r=1$.     
\end{proposition}
\begin{proof}
Let $\gamma: X\to S$ be the family. Attached to $\gamma$, we have the weight one $\Z$-VHS $V_{\Z}=R^1\gamma_*\Z_X$. Set $V=V_{\Z}\otimes_{\Z}\sO_S$. Then $(V,F_{hod},\nabla^{GM})$ is the degree one Gauss-Manin system over $S$ and $(E,\theta^{KS})=Gr_{F_{hod}}(V,\nabla^{GM})$ is the associated Kodaira-Spencer system over $S$. Let us describe $\nabla_{GM}$ on $f: M_{dR}(X/S,1)\to S$ as well as $\theta_{KS}$ on $g: M_{Hig}(X/S,1)\to S$ explicitly.

Via Riemann-Hilbert correspondence, $(f,\nabla_{GM})$ is isomorphic to $$(h: M_{Betti}(X/S,1)\to S,\nabla_{GM})$$ as holomorphic connection, where $\nabla_{GM}$ over $h$ comes from viewing $\gamma$ as a differentiable fiber bundle (viz. locally over $S$ as a product of differentiable manifolds). The holomorphic fibration $h$ is nothing but $V/V_{\Z}\to S$, a holomorphic $(\C^*)^{b_1}$-bundle over $S$, where $b_1$ is the first Betti number of $X_s$. It is also clear that $\nabla_{GM}$ is induced from $\nabla^{GM}$: The problem is local. We may assume $\gamma: X\to S$ to be a product as differentiable fiber bundle. Let $0\in S$ be a base point. Then $V_{\Z}\cong H^1(X_0,\Z)\times S$ and $V\to S$ is isomorphic to the trivial bundle with fiber $H^1(X_0,\C)$. $\nabla^{GM}$ is just the constant connection. Any holomorphic cross section of $V/V_{\Z}\to S$ lifts to a holomorphic section of $V\to S$ via the natural projection $V\to V/V_{\Z}$. Two liftings differ by an element of $H^1(X_0,\Z)$, viewed as a constant section of $V_{\Z}$ over $S$ which has value zero under $\nabla^{GM}$. So the value of $\nabla^{GM}$ on liftings of a holomorphic cross section is the same. The descending connection on $V/V_{\Z}\to S$ is $\nabla_{GM}$. 

Next, we write $E=E^{1,0}\oplus E^{0,1}$. By the $E_1$-degeneration of the Hodge to de Rham spectral sequence on the relative holomorphic de Rham complex, we know
$$
E^{1,0}=\gamma_*\Omega_{X/S},\quad E^{0,1}=R^1\gamma_*\sO_{X}.
$$
And we also know that the graded Higgs field 
$$
\theta^{KS}: T_S\to \Hom(E^{1,0}, E^{0,1})
$$
is just the composite of the Kodaira-Spencer map $\rho: T_S\to R^1\gamma_*T_{X/S}$ and the natural map
$$
\eta: R^1\gamma_*T_{X/S}\to \Hom(\gamma_*\Omega_{X/S}, R^1\gamma_*\sO_{X})
$$
coming from cup product and the natural pairing $T_{X/S}\otimes \Omega_{X/S}\stackrel{(\ \cdot \ )}{\longrightarrow} \sO_X$. We know that $M:=M_{Higgs}(X/S,1)$ (as $S$-variety) is naturally identified with
$$
T^*_{Pic^0(X/S)}\cong \gamma_*\Omega_{X/S}\times Pic^0(X/S)\cong \gamma_*\Omega_{X/S}\times R^1\gamma_*\sO_{X}/V_{\Z}.
$$
Here the inclusion $V_{\Z}\subset R^1\gamma_*\sO_{X}$ is the composite map
$V_{\Z}\to V\to E^{0,1}$. So we have a commutative diagram of holomorphic fibrations over $S$
\[
		\xymatrix{ E=E^{1,0}\times E^{0,1}\ar[dr]_h\ar[rr]^{id\times \pi}  & & E^{1,0}\times \frac{E^{0,1}}{V_{\Z}}=M \ar[dl]^{g} \\
			  & S. & }
		\]
The associated Higgs field to $\theta^{KS}$ to $h$ (see Example \ref{complex conjugation for vector bundles}) is invariant under translation by $E$ (in particular by $V_{\Z}\subset E$), and therefore descends to a Higgs field on $g$, which is again denoted by $\theta^{KS}$. Claim that $\theta^{KS}=\theta_{KS}$ as defined in Example \ref{nonabelian Kodaira-Spencer}. Set $N_1=X\times_SPic^0(X/S)$, $N_2=E^{1,0}$, and $N=X\times_SM$. Then we have the following Cartesian diagram
\[
		\xymatrix{ N\ar[d]_{p_1}\ar[r]^{p_2}  & N_2\ar[d]^{\pi_1} \\
			 N_1\ar[r]^{\pi_2} & S.  }
		\]
Let $\sP_{X/S}$ be the Poincar\'e line bundle over $N_1$ and $t\in H^0(N_2,\pi_1^*E^{1,0})$ the tautological section (the value of $t$ at $n\in N_2$ is $(n,n)\in N_2\times_S N_2=\pi_1^*E^{1,0}$). Then the universal Higgs bundle $(\sE,\Theta)$ over $N$ is the tensor product of $(\sP:=p_1^*\sP_{X/S},0)$ and $(\sO_N,\Theta)$, where $\Theta\in H^0(N,\Omega_{N/M})$ is the one obtained by sending $t$ under the composite of a sequence of natural maps
$$
H^0(N_2,\pi_1^*E^{1,0})=H^0(N_2,\pi_1^*\gamma_*\Omega_{X/S})\to H^0(N_2,\pi_1^*\gamma_*q_*q^*\Omega_{X/S}),
$$
where $q: N_1\to X$ is the natural projection satisfying $\pi_2=\gamma\circ q, g'=q\circ p_1$, and 
$$
H^0(N_2,\pi_1^*\gamma_*q_*q^*\Omega_{X/S})= H^0(N_2,\pi_1^*\pi_{2*}q^*\Omega_{X/S})\to H^0(N_2,p_{2*}p_1^*q^*\Omega_{X/S}),
$$
and 
$$
H^0(N_2,p_{2*}p_1^*q^*\Omega_{X/S})=H^0(N_2,p_{2*}g^{'*}\Omega_{X/S})= H^0(N_2,p_{2*}\Omega_{N/M})=H^0(N, \Omega_{N/M}).
$$
According to the above description, the Higgs bundle $(\sE nd(\sE),\Theta^{end})$ is simply given by
$$
\Theta: \sO_N\to \sO_N\otimes \Omega_{N/M}=\Omega_{N/M},
$$
and therefore the associated Higgs complex $\Omega_{Hig}(\sO_N,\Theta)$ reads
$$
\sO_N\stackrel{\Theta}{\to} \Omega_{N/M}.
$$
The verification of the claim boils down to showing $\tau: R^1\gamma_*T_{X/S}\to g_*T_{M/S}$ coincides with $\eta$ under the natural quotient map $\pi$. It is a local problem. So we take an open subset $V\subset S$ and an open subset $U=U_1\times U_0\subset g^{-1}(V)\subset M$, so that $\pi: \pi^{-1}U_0 \cong U_0$, and $\tau$ is given by 
$$
H^1(X_V, T_{X_V/V})\stackrel{g^{'*}}{\to} H^1(N_U,T_{N_U/U})\to \H^1(N_U, \Omega^*_{Hig}(\sO_{N_U},\Theta_U))\cong H^0(U,T_{U/V}),
$$
where $X_V=\gamma^{-1}V$, $N_U=\gamma^{'-1}(U)$ and $(\sO_{N_U},\Theta_U)$ is the restriction of $(\sO_N,\Theta)$ to $N_{U}$. Take an open covering $\sU=\{U_{\alpha}\}$ of $X_V$, so that $\{W_{\alpha}:=U_{\alpha}\times_{V}U\}$ forms an open covering $\sW$ of $N_U$. Elements of $\H^1(N_U, \Omega^*_{Hig}(\sO_{N_U},\Theta_U))$ are represented by cech cocycles in $\sC^0(\sW,\Omega_{N_U/U})\oplus \sC^1(\sW,\sO_{N_U})$. Represent an element of $H^1(X_V, T_{X_V/V})$ by $\nu_{\alpha_0\alpha_1}\in \sC^1(\sU,T_{X_V/V})$. Then the action of $g^{'*}\nu_{\alpha_0\alpha_1}\in \sC^1(\sW,T_{N_U/U})$ on the first summand $\gamma^0_{\alpha_0}\in \sC^0(\sW,\Omega_{N_U/U})$ is given by $(g^{'*}\nu_{\alpha_0\alpha_1},\gamma^0_{\alpha_0}|_{W_{\alpha_0\alpha_1}})$, and on the second summand $\gamma^1_{\alpha_0\alpha_1}\in \sC^1(\sW,\sO_{N_U})$ is given by $\gamma^1_{\alpha_0\alpha_1}\cdot(g^{'*}\nu_{\alpha_0\alpha_1},\Theta|_{W_{\alpha_0\alpha_1}})$. As $\Theta$ is a global section of $\Omega_{N/M}$, the element $\gamma^1_{\alpha_0\alpha_1}\cdot(g^{'*}\nu_{\alpha_0\alpha_1},\Theta|_{W_{\alpha_0\alpha_1}})$ is a coboundary. Therefore, the action of $H^1(X_V, T_{X_V/V})$ on 
$$
H^0(U,T_{U/V})\cong H^0(U_1\times \pi^{-1}U_0,T_{U_1\times \pi^{-1}U_0/V})\cong h_1^*E^{1,0}|_{U_1}\boxplus h_0^*E^{0,1}|_{U_0},  
$$
where $h_i: E^{i,1-i}\to S$ is is the natural map, coincides with that of $\theta^{KS}$. 

Third, Conjecture \ref{identification of differentiable structrues} in this case is trivial. Let $h_{hod}$ be the Hodge metric on $E$. By Example \ref{hermitian  metrics on vector bundles}, we obtain a hermitian metric $\omega$ on $h: E\to S$. As $\{d e^*_i\}$s are translation-invariant, $\omega$ descends to a hermitian metric on $g: M\to S$ (denoted again by $\omega$), which is easily seen to be a Hitchin metric. 

Finally, by observing that the verification of the vanishing of $F_{\omega}$ is local on $M$, and hence it can be done on $E$ by choosing a local isomorphism as above. The pullback of $F_{\omega}$ to $E$ is nothing but the curvature associated to $h$ over $(E,\theta^{KS})$, which vanishes because $h$ is harmonic for it. 

\end{proof}

\section{A Torelli theorem}
To conclude this note, we would like to apply our preliminary theory to the theory of abelian varieties, and show that we do gain some new information from such considerations. Let $A_0$ be a complex abelian variety. The Torelli theorem for abelian varieties says that $A_0$ up to isomorphism is determined by its associated weight one $\Z$-PHS on $V_{\Z}=H^1(A_0,\Z)$. By Dolbeault isomorphism, one obtains from $V_{\Z}$ the purely algebraically defined cohomology $H^{0}(A_0,\Omega_{A_0})\oplus H^1(A_0,\sO_{A_0})$. The process is non reversible. However, when $A_0$ is put into a family which has large monodromy action on $H^1$, then the associated Kodaria-Spencer system with fiber at $0$ the above Dolbeault cohomologies together with an action of Higgs field will tell much more information about the Hodge structure on $V_{\Z}\otimes \R$. But again, one cannot completely determine $V_{\Z}$.   
\begin{theorem}
Let $S$ be a projective manifold of positive dimension. Let $\sA^i\to S, i=1,2$ be two polarized families of abelian varieties over $S$ whose associated monodromy representations on $H^1$ are Zariski dense. Then these two polarized families are isomorphic to each other if and only if the associated nonlinear Higgs bundles $ (M_{Hig}(\sA^i/S,1)\to S,\theta_{KS})$ over $S$ are isomorphic to each other.
\end{theorem}
\begin{proof}
It suffices to prove the if-direction. A polarized family of abelian varieties over $S$ is uniquely determined up to isomorphism by its associated weight one $\Z$-PVHS over $S$. Fix a base point $0\in S$. Let $\rho^i: \pi_1(S,0)\to \Aut(H^1(\sA^i_0,\Z),\omega^i_0)$ be the associated monodromy representation, where $\omega^i_0$ denotes for the symplectic form coming from polarization $V_{\Z}\times V_{\Z}\to \Z$. The condition on Zariski density says that the Zariski closure of $\rho^i_{\C}$ in the algebraic group $\Aut(H^1(\sA^i_0,\C),\omega^i_0)$ is the whole group. Let $\V^i_{\Z}$ be the corresponding $\Z$-local system to $\rho^i$. By the Simpson correspondence, the corresponding Higgs bundles $(E^i,\theta^{KS})$ over $S$ are stable (with respect to any ample line bundle over $S$). Moreover, $\wedge^2\rho^i_{\C}\cong \wedge^2_0\rho^i_{\C}\oplus \C\{\omega_0^i\}$ decomposes into two irreducibles. Set $\V^i_{\C}=\V^i_{\Z}\otimes \C$ and $V^i=\V^i_{\Z}\otimes \sO_S$. Then the nonabelian Hodge theory tells us that $(E^i,\theta^{KS})$ determines $\V^i_{\C}$ as complex variation of Hodge structure. As $(M_{Hig}(\sA^1/S,1),\theta_{KS})\cong (M_{Hig}(\sA^2/S,1),\theta_{KS})$, it follows from the proof of Proposition \ref{rank one moduli} that 
$$(M_{Betti}(\sA^1/S,1)\to S,\nabla_{GM})\cong (M_{Betti}(\sA^2/S,1)\to S,\nabla_{GM}).$$ Also, the proof explains that $M_{Betti}(\sA^i/S,1)\to S$ is isomorphic to the quotient $V^i/\V^i_{\Z}\to S$, so that we obtain that as local systems over $S$,
$$
\V^i_{\Z}=\ker(Lie(M_{Betti}(\sA^1/S,1))\to M_{Betti}(\sA^1/S,1)).
$$
Hence we see that in the isomorphism $\varphi: \V^1_{\C}\to \V^2_{\C}$ of $\C$-VHS, it maps $\V^1_{\Z}$ to $\V^2_{\Z}$. In other words, we have an isomorphism of weight one $\Z$-VHS $\V^1_{\Z}\to \V^2_{\Z}$. But since the $\pi_1(S,0)$-invariant $\C$-subspace in $\wedge^2\V^i_{\C}$ is one-dimenionsal, it follows that the $\pi_1(S,0)$-invariant $\Z$-sublattice in $\wedge^2\V^i_{\Z}$ is of rank one, and hence the induced isomorphism $\wedge^2\V^1_{\Z}\to \wedge^2\V^2_{\Z}$ of $\Z$-local systems must map $\Z\{\omega^1_0\}$ to $\Z\{\omega^2_0\}$, or equivalently map $\omega_0^1$ to $\pm\omega^2_0$. Therefore, either $\varphi$ or $-\varphi$ induces an isomorphism $\V^1_{\Z} \to \V^2_{\Z}$ of weight one $\Z$-PVHS. This concludes the proof. 

\end{proof}


\begin{thebibliography}{5}
 
\bibitem[AL]{AL}
E. Anders\'en, L. Lempert, On the group of holomorphic automorphisms of $\C^n$, Invent. Math. 110, (1992), 371-388.  

\bibitem[At]{At}
M. F. Atiyah, Complex analytic connections in fiber bundles, Transactions of the American Mathematical Society, Vol. 85, No. 1, pp. 181-207.

\bibitem[C]{C}
S. S. Chern, Differential geometry of fiber bundles, Proceedings of the International Congress of Mathematicians, 1950, pp. 397-411.

\bibitem[H]{H}
N. Hitchin, The self-duality equations on a Riemann surface, Proc. Longdon Math. Soc. (3) 55 (1987) 59-126.

\bibitem[OV]{OV}
A. Ogus, V. Vologodsky, Nonabelian Hodge theory in characteristic $p$, Publications math\'ematiques de l'I.H.\'E.S., tome 106 (2007), p. 1-138. 

\bibitem[FS]{FS}
Y.-X. Fu, M. Sheng, Nonabelian Kodaira-Spencer maps, arXiv:2509.06050, Sep. 2025.

\bibitem[S]{S}
M. Sheng, Nonlinear Hodge theory in positive characteristic, arXiv:2510.05578, Oct. 2025. 

\bibitem[Si1]{Si1}
C. Simpson, Higgs bundles and local systems, Publications math\'ematiques de l'I.H.\'E.S., tome 75 (1992), p. 5-95. 

\bibitem[Si2]{Si2}
C. Simpson, Moduli of representations of the fundamental group of a smooth projective variety II, Publications math\'ematiques de l'I.H.\'E.S., tome 80 (1994), p. 5-79. 

\bibitem[Si3]{Si3}
C. Simpson, The Hodge filtration on nonabelian cohomology, Proceedings of symposia in pure mathematics, Vol. 62, 1997. 

\end{thebibliography}
\end{document}